\documentclass[12pt,reqno]{amsart}

\usepackage[T1]{fontenc}
\usepackage[english]{babel}
\usepackage{newcent}

\usepackage{amsmath,amssymb,amscd,amsfonts,tikz,caption}
\usetikzlibrary{patterns}

\usepackage{graphicx}
\usepackage{xcolor}
\usepackage[linkcolor=blue,colorlinks=true,citecolor=blue,filecolor=black]{hyperref}

\newtheorem{Thm}{Theorem}
\newtheorem{Coro}{Corollary}

\newtheorem{thm}{Theorem}[section]
\newtheorem{lm}[thm]{Lemma}

\newtheorem{prop}[thm]{Proposition}

\numberwithin{equation}{section}

\newcommand{\abs}[1]{\left\vert#1\right\vert}

\newcommand{\R}{\mathbb{R}}
\newcommand{\C}{\mathbb{C}}
\newcommand{\h}{\mathbb{H}}
\newcommand{\Q}{\mathbb{Q}}
\newcommand{\Z}{\mathbb{Z}}
\newcommand{\N}{\mathbb{N}}
\newcommand{\eps}{\varepsilon}
\newcommand{\im}{\mathfrak{Im}}
\newcommand{\re}{\mathfrak{Re}}

\newcommand{\bpm}{\begin{pmatrix}}
\newcommand{\epm}{\end{pmatrix}}
\newcommand{\bsm}{\left(\begin{smallmatrix}}
\newcommand{\esm}{\end{smallmatrix}\right)}

\newcommand{\sign}{\mathrm{sign}}
\newcommand{\vol}{\mathrm{vol}}

\newcommand{\SL}{\mathrm{SL}}
\newcommand{\PSL}{\mathrm{PSL}}
\newcommand{\res}{\mathrm{Res}}

\title{Generalized Dedekind sums and equidistribution mod 1}
\author{Claire Burrin}

\begin{document}
\baselineskip=18pt

\maketitle

\begin{abstract}
Dedekind sums are well-studied arithmetic sums, with values uniformly distributed on the unit interval. Based on their relation to certain modular forms, Dedekind sums may be defined as functions on the cusp set of $\SL(2,\Z)$. We present a compatible notion of Dedekind sums, which we name Dedekind symbols, for any non-cocompact lattice $\Gamma<\SL(2,\R)$, and prove the corresponding equidistribution mod 1 result. The latter part builds up on a paper of Vardi, who first connected exponential sums of Dedekind sums to Kloosterman sums.
\end{abstract}

\setcounter{tocdepth}{1}
\tableofcontents

\section{Introduction}

In this note, we introduce a function for non-cocompact lattices of $\SL(2,\R)$ that relates to, and actually generalizes, the classical Dedekind sums 
$$
s(a;c)\ =\ \sum_{n=1}^{c-1} \left(\left(\frac{n}{c}\right)\right)\left(\left(\frac{na}{c}\right)\right)\qquad (c\in\N, a\in\Z, (a,c)=1),
$$
where 
$$
x \mapsto \left(\left( x\right)\right)\ :=\ \begin{cases} \left\{x\right\} - \frac{1}{2} & x\not\in\Z \\ 0 & x \in\Z \end{cases}\qquad \ (\{ x\} = \text{ fractional part of } x\in\R)
$$
is the odd and periodic "sawtooth" function of expectancy zero, with graph 

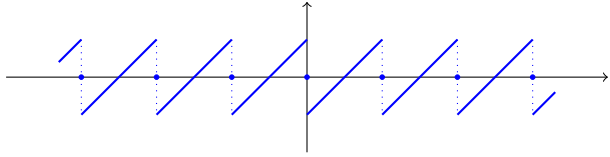
\begin{center}
\begin{tikzpicture}
\draw[->] (-4,0)--(4,0);
\draw[->] (0,-1)--(0,1);
\draw[blue,thick] (0,-.5)--(1,.5);
\draw[blue,dotted] (1,.5)--(1,-.5);
\draw[blue,dotted] (2,.5)--(2,-.5);
\draw[blue,dotted] (3,.5)--(3,-.5);
\draw[blue,dotted] (-1,.5)--(-1,-.5);
\draw[blue,dotted] (-2,.5)--(-2,-.5);
\draw[blue,dotted] (-3,.5)--(-3,-.5);
\draw[blue,thick] (1,-.5)--(2,.5);
\draw[blue,thick] (2,-.5)--(3,.5);
\draw[blue,thick] (-1,-.5)--(0,.5);
\draw[blue,thick] (-3,-.5)--(-2,.5);
\draw[blue,thick] (-2,-.5)--(-1,.5);
\draw[blue,thick] (-3,.5)--(-3.3,.2);
\draw[blue,thick] (3,-.5)--(3.3,-0.2);
\draw[fill,blue] (0,0) circle[radius=0.03];
\draw[fill,blue] (1,0) circle[radius=0.03];
\draw[fill,blue] (2,0) circle[radius=0.03];
\draw[fill,blue] (3,0) circle[radius=0.03];
\draw[fill,blue] (-1,0) circle[radius=0.03];
\draw[fill,blue] (-2,0) circle[radius=0.03];
\draw[fill,blue] (-3,0) circle[radius=0.03];
\end{tikzpicture}
\end{center}

There is a ubiquitous character to these sums, for the wide range of contexts they appear in. The name Dedekind sums hinges on their relation to the logarithm of the Dedekind $\eta$-function
\begin{align*}
\eta(z)\ =\ e^{\frac{\pi iz}{12}}\prod_{n\geq1}\left(1-e^{2\pi inz}\right)
\end{align*} 
defined on the upper half plane $\h$, a classical player in the theories of modular forms, elliptic curves, and theta functions. More precisely,  for every $\gamma=\bsm a&b\\c&d\esm\in\SL(2,\Z)$,
\begin{align}\label{DTF}
\log\eta(\gamma z) - \log\eta(z)\ =\ \frac{1}{2}(\sign c)^2 \log\left( \frac{cz+d}{i\sign c}\right) + \frac{\pi i}{12} \Phi(\gamma) 
\end{align}
where the defect $\Phi(\gamma)$ arising from the ambiguity of the (principal branch of the) logarithm is given by
\begin{align}\label{Phi def}
\Phi(\gamma) \ =\ \begin{cases} b/d & c=0 \\ \frac{a+d}{c} - 12\mathrm{sign}(c) s(a;\abs{c}) & c\neq 0.
\end{cases}
\end{align}
While this is not obvious at first glance, the values of $\Phi$ are always integers. The latter fact, as many other fundamental properties pertaining to Dedekind sums, may be found in the monograph \cite{Rad}. While Dedekind's original proof of the transformation formula for $\log\eta$ above is of analytic nature, it can also be deduced by purely topological arguments. Atiyah's paper \cite{Ati} discusses this approach and offers an overview of the appearance of $\log\eta$ and the Dedekind sums in various contexts of number theory, topology and geometry. In particular, Atiyah exhibits no less than seven equivalent characterizations of $\log\eta$ across these different fields.

An alternative presentation of the Dedekind sums consists in defining $s(a;c)$ as a function on the cusp set of $\SL(2,\Z)$, which can be identified with the extended rational line $\Q\cup\{\infty\}$. This identification can then be exploited to study some of their related properties via continued fraction expansions, as is done in \cite{KM}.

We propose a modified construction. Let $\Gamma_\infty$ denote the stabilizer subgroup of $\Gamma=\SL(2,\Z)$ at $i\infty$, that is,
$$
\Gamma_\infty\ =\ \left\{ \bpm *&*\\&*\epm\in\SL(2,\Z)\right\}\ =\ \pm\bpm 1&\Z\\&1\epm.
$$
There is a one-to-one correspondence between the cusp set of $\Gamma$, i.e.
$$
\{\gamma(i\infty) : \gamma\in\Gamma\}
$$ and the quotient $\Gamma\slash\Gamma_\infty$. We can thus express (signed) Dedekind sums via the assignment
\begin{align*}
\bpm a & b \\ c & d\epm \Gamma_\infty \ \mapsto\  \frac{1}{12}\frac{a+d}{c} - \frac{1}{12}\Phi\bpm a & b \\ c & d\epm\quad \left(\overset{(\ref{Phi def})}{=}\ \sign(c)s(a;\abs{c})\right). 
\end{align*}
This map descends to the double coset $\Gamma_\infty \backslash\ \Gamma\slash\ \Gamma_\infty$. In fact, this is simply a manifestation of the periodicity of the Dedekind sums, since, for each integer $m$, 
$$
\bpm 1&m\\&1\epm\bpm a&b\\ c&d\epm\ =\ \bpm a+ mc& b+md\\ c&d\epm
$$
and
$$
\sign(c)s(a+mc;\abs{c})\ =\ \sign(c)s(a;\abs{c}).
$$
We call the resulting double coset function 
$
\mathcal{S}: \Gamma_\infty \backslash\ \Gamma\slash\ \Gamma_\infty \to\Q
$
the Dedekind symbol for $\SL(2,\Z)$. 

This construction may be generalized to any non-cocompact lattice $\Gamma<\SL(2,\R)$. For simplicity, let us assume for the rest of this introduction that our preferred cusp for $\Gamma$ is at $i\infty$ and that the corresponding stabilizer subgroup  is $\Gamma_\infty = \bsm *&*\\&*\esm\cap\Gamma =  \pm \bsm 1&\Z\\&1\esm$. 

\begin{Thm}\label{Dedekind symbols}
Let $\Gamma<\SL(2,\R)$ be a non-cocompact lattice. Then
\begin{enumerate}
\item There exists a continuous family $\{f_{k}\}_{k\in\R}$ of holomorphic nowhere vanishing functions on $\h$ that each transform with respect to the action of $\Gamma$ by
\begin{align*}%\label{MTF}
\qquad \log f_{k}(\gamma z) - \log f_{k}(z)\ =\ k\log\left(\frac{cz+d}{i\sign(c(-d))}\right) + 2\pi ik\phi(\gamma) 
\end{align*}
where $\phi$ is a real-valued function, $\log$ denotes the principal branch of the logarithm and $c(-d)=c$ if $c\neq 0$ and $-d$ otherwise.
\item   Moreover, for each integer $m$,
$$
\phi\left(\pm\bpm1&m\\&1\epm\right)\ =\ \frac{\mathrm{vol}(\Gamma\backslash\h)}{4\pi}m -\frac{1}{4},
$$
and the double coset function
\begin{align*}%\label{Dedekind symbol def}
\mathcal{S}\left(\Gamma_\infty\bpm a&*\\c&d\epm\Gamma_\infty\right)\ =\ \begin{cases} \frac{\mathrm{vol}(\Gamma\backslash\h)}{4\pi}\frac{a+d}{c} - \phi(\gamma) &c\neq 0\\ \infty & c=0
\end{cases}
\end{align*}
is well-defined. We call $\mathcal{S}:\Gamma_\infty \backslash\ \Gamma\slash\ \Gamma_\infty\to\R$ the Dedekind symbol for $\Gamma$ at its cusp at infinity.
\item  Let $\Gamma=\SL(2,\Z)$. Then $f_{1/2}(z)$ coincides with Dedekind's $\eta$-function, and $$\mathcal{S}\left(\Gamma_\infty\bsm a&*\\c&d\esm\Gamma_\infty\right) = \sign(c)s(a;\abs{c}).$$
\end{enumerate}
\end{Thm}

Goldstein \cite{Go,Go2} derived formally the functions $f_k$ from the Fourier expansion of Eisenstein series, and used them to give explicit formulas of Dedekind sums for certain principal congruence subgroups. Our approach differs from Goldstein's in that it does not rely on explicit Fourier coefficients and that we moreover prove the analytic existence of the functions $f_k$. Furthermore, the definition of Dedekind symbols as double coset functions is new.

We note that the function $f_k$ arises in association with the fixed cusp at $i\infty$. In the case of the full modular group, this is the only cusp, but for a more general group $\Gamma$, one can construct as many families $\{f_k\}$ as there are (inequivalent) cusps for $\Gamma$. We deduce from their construction the existence of a unique cusp form, natural analogue of the Dedekind $\eta$-function for a cofinite Fuchsian group.

\begin{Coro}\label{cusp form thm}
Let $\Gamma<\SL(2,\R)$ be a non-cocompact lattice. Then there exists a holomorphic, nowhere-vanishing real-weight cusp form $\eta_\Gamma$ for any positive real weight $k$, which generalizes Dedekind's $\eta$-function.
\end{Coro}

The second part of the paper concerns the distribution of values of the Dedekind symbol $\mathcal{S}$. 
The statistics of Dedekind sums have been extensively studied; we know that their values become equidistributed mod 1  \cite{Var}, and that this result extends to the graph $\left(\frac{a}{c},s(a;c)\right)$ \cite{Myer}. Bruggeman studied the distribution of $s(a;c)/c$ \cite{Brug} and Vardi showed that $s(a;c)/\log c$ has a limiting Cauchy distribution as $c\to\infty$ \cite{Vardi2}. The focus later shifted to the distribution of mean values of Dedekind sums \cite{CFKS,Zhang}.

Here, we will be interested in the problem of equidistribution mod 1 of values of the Dedekind symbol. Recall that a real sequence $( a_n)_{n\in\N}$ is said to become equidistributed mod 1 if its fractional parts are uniformly distributed. That is, for any subinterval $[a,b]\subset [0,1]$, if
$$
\lim_{n\to\infty}\frac{\#\{ i=1,\dots,n : \{a_i\}\in [a,b]\}}{n}\ =\ b-a,
$$
where $\{a_i\}$ denotes the fractional part of $a_i$. Weyl \cite{Wey} famously formulated an equivalent criterium for equidistribution mod 1 in terms of exponential sums: the sequence $(a_n)_{n\in\N}$ becomes equidistributed mod 1 if and only if, for every non-zero integer $m$, 
\begin{align*}
\sum_{n\leq N} e(ma_n)\ =\ o(N)
\end{align*}
where $e(x):= e^{2\pi ix}$.

Our interest in the question was motivated by Vardi's proof of the following strong form of equidistribution mod 1 for Dedekind sums. For any $k\in\R_{>0}$, the sequence of values 
$$
\left\{ k s(a;c)\right\}_{\substack{0\leq a<c \\ (a,c)=1}}
$$
becomes equidistributed mod 1 as $c\to\infty$ \cite[Thm. 1.6]{Var}. The building block of Vardi's proof is a striking identity relating Dedekind sums to Kloosterman sums. A simplified version goes as follows. 

Kloosterman sums, introduced as a refinement of the Hardy--Littlewood circle method, are defined by
$$
S(m,n;c)\ =\ \sum_{\substack{a=1\\ (a,c)=1\\ ad\equiv 1\ \text{mod}\ c}}^{c-1} e\left(\frac{ma+nd}{c}\right).
$$
Now Vardi observed the following identity; for any $m\in\N$,
\begin{align*}
\sum_{\substack{a=1\\ (a,c)=1}}^{c-1} e\left( 12m s(a;c)\right)\ &\stackrel{(\ref{Phi def})}{=}\ \sum_{\substack{a=1\\ (a,c)=1}}^{c-1} e\left(\frac{ma+md}{c}\right)\underbrace{e^{-2\pi im\phi(\gamma)}}_{\equiv 1}\ =\ S(m,m;c).
\end{align*}
It then follows, by Weyl's criterium, that the values along $\{12s(a;c)\}$ become equidistributed mod 1 as $c\to\infty$ if and only if, for each $m\in\N$, 
$$\sum_{c\leq x} S(m,m; c) = o(x^2).$$
That is, in other words, given enough cancellation in sums of Kloosterman sums. Such estimates exist, and a very strong variant of such estimates is provided by Weil's 
$$
S(m,n;c)\ \ll\ c^{1/2+\eps} \qquad (\eps>0).
$$
(The full force of Weil's estimate is absolutely not necessary here, see \cite{Var} or Theorem \ref{equidistribution mod 1} below.) Vardi then goes on to show that this holds true for any positive multiplicative scalar $k$, relying on the spectral theory of automorphic forms and work of Goldfeld--Sarnak \cite{GS}.

To approach the question of the equidistribution of Dedekind symbols, we first need to order double cosets in $\Gamma_\infty\backslash\Gamma\slash\Gamma_\infty$. Such a parametrization is encoded in the double coset decomposition
$$
\Gamma\ =\ \Gamma_\infty\cup\bigcup_{c>0}\bigcup_{\substack{0\leq a<c\\ \bsm a&*\\c&*\esm\in\Gamma}}\Gamma_\infty\bpm a&*\\c&*\epm\Gamma_\infty.
$$
(This is reviewed for convenience in Section 2.) 

\begin{Thm}\label{counting thm}
Let $\Gamma<\SL(2,\R)$ be a non-cocompact lattice, and let $x>0$. The double coset count 
$$
\pi(x)\ =\ \# \left\{ \Gamma_\infty\bpm a&*\\c&*\epm \Gamma_\infty : \substack{ 0< c<x\\ 0\leq a<c,}\  \bsm a&*\\c&*\esm\in\Gamma  \right\}
$$ is finite, and
$$
\pi(x)\ \sim\ \frac{x^2}{\pi \vol(\Gamma\backslash\h)}
$$
as $x\to\infty$.
\end{Thm}

The main result of this paper is the generalization of Vardi's strong form of equidistribution to Dedekind symbols for cofinite Fuchsian groups.

\begin{Thm}\label{equidistribution mod 1}
For any $k\in\R_{>0}$, the sequence of values 
$$
\left\{k\mathcal{S}(\Gamma_\infty\bsm a&*\\c&*\esm\Gamma_\infty)\right\}_{\substack{0\leq a<c\\ \bsm a&*\\c&*\esm\in\Gamma}}
$$ 
becomes equidistributed mod 1 as $c\to\infty$.
\end{Thm}

Our proof is intrinsically similar to that of \cite{Var}. In fact, we also recover a type of Vardi identity in terms of Kloosterman sums twisted by a multiplier system (cf.~Proposition \ref{prop Vardi identity}). Such sums were introduced by Selberg \cite{Sel} to estimate the order of magnitude of Fourier coefficients of cusp forms, and the work of Goldfeld--Sarnak \cite{GS} provides us with an estimate of the growth of sums of such Kloosterman sums. The new difficulty, at this level of generality, is that one needs precise control on the growth of double cosets. This extra gap is bridged by the content of Theorem \ref{counting thm}.

\section{Preliminaries}

This section reviews the necessary facts and results on Fuchsian groups, Eisenstein series, real weight automorphic forms, ordering of double cosets, and Selberg's generalized Kloosterman sums.

\subsection*{Non-cocompact lattices of $\SL(2,\R)$}
 
Let $\Gamma<\SL(2,\R)$ be a lattice, with the assumption that $-I\in\Gamma$. The projection $\overline{\Gamma}<\PSL(2,\R)$ acts properly discontinuously on the upper-half plane $\h$ by linear fractional transformations
$$
\bpm a&b\\c&d\epm\ :\ z\ \mapsto\ \frac{az+b}{cz+d}.
$$
 If moreover $\Gamma$ is non-cocompact, then $\Gamma$ admits a finite number of inequivalent cusps. A cusp $\frak{a}\in\partial\h = \R\cup\{\infty\}$ is a parabolic fixed point for the extended action of $\overline{\Gamma}$ on $\partial\h$. Two cusps $\frak{a}$ and $\frak{b}$ are said to be equivalent if $\gamma\frak{a}=\frak{b}$ for some $\gamma\in\Gamma$. In practice, it is most useful to work with the cusp at $\infty$. There is a standard change of coordinates to achieve this. In fact, for each cusp $\frak{a}$, there exists a scaling matrix $\sigma_\frak{a}\in\SL(2,\R)$ that verifies
\begin{enumerate}
\item $\sigma_\frak{a}(\infty) =\frak{a}$
\item $\sigma_\frak{a}^{-1}\Gamma_\frak{a}\sigma_\frak{a} = \left(\sigma_\frak{a}^{-1}\Gamma\sigma_\frak{a}\right)_\infty = \pm\bpm 1 & \Z \\ & 1\epm$
\end{enumerate}
These two conditions do not determine a scaling matrix uniquely, but up to right multiplication by any matrix of the form $\pm\bsm1&x\\&1\esm$, $x\in\R$ \cite[p. 40]{Iwa}.

\subsection*{Double coset decomposition}
Let 
$$
T\ :=\ \pm \bpm 1&\Z\\&1\epm.
$$
The trivial computation
$$
\bpm 1 & m \\ & 1 \epm\bpm a & b \\ c & d \epm\bpm 1 & n \\ & 1 \epm\ =\ \bpm a + mc & * \\ c & d+nc \epm,
$$
shows that the lower left matrix entry $c$ depends only on the double coset $T\gamma T$, and that $a$ and $d$ are determined up to integer multiples of $c$. Each double coset for which $c\neq0$ has then a unique representative of the form
 $$
T\bpm a & * \\ c & d \epm T,\qquad c>0,\quad 0\leq a,\ d< c.
$$
Moreover, for $\gamma= \bsm a & b \\ c & d\esm$ and $\gamma' = \bsm a & b' \\ c & d'\esm$ two matrices of determinant 1, one has
\begin{align*}
\gamma^{-1}\gamma'\ =\ \bpm 1 & * \\ 0 & 1 \epm.
\end{align*}
Therefore, any double coset $T\gamma T$ for which $c\neq0$ is really only determined by the left column of the representative $\gamma$. On the other hand, $\bsm *&*\\&*\esm\cap\left(\sigma_\frak{a}^{-1}\Gamma\sigma_\frak{a}\right)=T$. In conclusion,
\begin{align*}%\label{DCD}
\sigma_\frak{a}^{-1}\Gamma\sigma_\frak{a}\ =\ T \cup \bigcup_{c>0}\bigcup_* T\bpm a & * \\ c & *\epm T
\end{align*}
where the union $\bigcup_*$ is taken over all $\bpm a & * \\ c & *\epm\in\sigma_\frak{a}^{-1}\Gamma\sigma_\frak{a}$ with $0\leq a< c$. 
For any $x>0$, there are at most finitely many double cosets $T\bsm a&*\\c&d\esm T$ such that $\bsm a&*\\c&d\esm\in\sigma_\frak{a}^{-1}\Gamma\sigma_\frak{a}$ and $\abs{c}\leq x$ \cite[Lm. 1.24]{Shi}.

\subsection*{Eisenstein series}
The Eisenstein series for $\Gamma$ at its cusp at $\frak{a}$ is defined by
\begin{align*}
E_\frak{a}(z,s)\ :=\ E_\frak{a}(z,s;\Gamma)\ =\ \sum_{\gamma\in \Gamma_\frak{a}\backslash\Gamma} \im(\sigma_\frak{a}^{-1}\gamma z)^s 
\end{align*}
where $z\in\h$, $s=\sigma +it\in\C$. The series converges absolutely and uniformly on compact subsets for $\sigma>1$. As a function of $z$, it is $\Gamma$-invariant, non-holomorphic and satisfies
$$
\Delta E_\frak{a}(z,s)\ =\ s(1-s) E_\frak{a}(z,s), 
$$ 
for the (positive) hyperbolic Laplacian $\Delta=-y^2\left(\partial_{xx} +\partial_{yy}\right)$. Eisenstein series admit a Fourier expansion in each cusp, which takes the form
\begin{align*}%\label{Fourier expansion Eisenstein series}
E_\frak{a}(\sigma_\frak{b}z,s)\ =\ \delta_\frak{ab}y^s + \varphi_\frak{ab}(s)y^{1-s} +O\left(e^{-2\pi y}\right)
\end{align*}
where 
\begin{align*}%\label{varphi}
\varphi_\frak{ab}(s)\ =\ \sqrt\pi\frac{\Gamma(s-1/2)}{\Gamma(s)}\sum_{c>0}\frac{\#\{  a\in[0,c) : \bsm a&*\\c&*\esm\in\sigma_\frak{a}^{-1}\Gamma\sigma_\frak{b}\} }{c^{2s}}
\end{align*}
\cite[Thm. 3.4]{Iwa}. In the definition above, $\Gamma(s)$ denotes the classical Gamma function, which is holomorphic on the complex plane except for simple poles at every non-positive integer.

Eisenstein series famously admit a meromorphic continuation to the whole complex $s$-plane, which follows from the meromorphic continuation of $\varphi_\frak{ab}$ \cite{STF}. In particular, $\varphi_\frak{ab}(s)$ is holomorphic in the half-plane $\sigma>1/2$ except for possibly finitely many simple poles $\sigma_j\in(1/2,1)$ and a simple pole at $s=1$ of residue $$\frac{1}{\vol(\Gamma\backslash\h)}.$$ Moreover, away from the real line, $\varphi_\frak{ab}(s)$ is bounded in the half-plane $\sigma>1/2$ \cite[Eq. (8.6)]{Sel2}.

\subsection*{Real weight automorphic forms}

Such an automorphic form is understood to be a holomorphic function on $\h$ that transforms by
$$
f(\gamma z)\ =\ \chi_k(\gamma) (cz+d)^k f(z),
$$
for each $\gamma\in\Gamma$, some fixed weight $k\in\R$ and where the so-called multiplier system $\chi_k(\Gamma)$ must be consistent with the given determination of $\arg(cz+d)$ such that 
$$
(cz+d)^k\ =\ \abs{cz+d}^k e^{ik\arg(cz+d)}
$$
is uniquely determined. Explicitly, we say that $\chi_k:\Gamma\to\C$ defines a {\it multiplier system of weight $k$ for $\Gamma$} if it satisfies the consistency conditions
\begin{enumerate}
\item For all $\gamma\in\Gamma$, $\abs{\chi_k(\gamma)}=1$
\item $\chi_k(-I) = e^{-\pi ik}$
\item $\chi_k(\gamma_1\gamma_2) j(\gamma_1\gamma_2,z)^k = \chi_k(\gamma_1)\chi_k(\gamma_2)j(\gamma_1,\gamma_2 z)^k j(\gamma_2,z)^k$
\end{enumerate}
where $j(\gamma,z):= cz+d$ and $\arg(z)\in(-\pi,\pi]$. The Dedekind $\eta$-function is an automorphic form of weight $1/2$ for $\SL(2,\Z)$ with respect to the multiplier system $$\chi_{1/2} = e^{\frac{\pi i}{12}\Phi}$$ with the additional property that it is nowhere zero on $\h$. More generally, the existence of real-weight automorphic forms, however not necessarily holomorphic, is guaranteed if $\Gamma$ has at least one cusp \cite[pp. 333-335]{Hej}.

\subsection*{Selberg's general Kloosterman sums}
Assume again that $\Gamma$ has a cusp at $i\infty$ and stabilizer subgroup $\Gamma_\infty = \pm \bsm 1&\Z\\&1\esm$. An automorphic form for $\Gamma$ will verify the periodicity relation 
$$
f(z+1)\ =\ \chi_k\bpm1&1\\&1\epm f(z)
$$
and thus admit a Fourier expansion of the form 
$$
f(z)\ =\ \sum c_n e^{2\pi i(n-\alpha)z}
$$
where $\alpha :=\alpha(\chi_k)$ denotes the unique scalar in $[0,1)$ such that $$\chi_k\bpm1&1\\&1\epm = e^{-2\pi i\alpha}.$$
The famous problem of estimating the order of magnitude of Fourier coefficients of a cusp form can be reduced to estimating the {\it generalized Kloosterman sums} 
\begin{align*}%\label{GKS}
S\left(m,n\ ; c, \chi_k\right)\ :=\ \sum_{*} \overline{\chi_k\bpm a&*\\c&d\epm}e\left( \frac{(m-\alpha)a + (n-\alpha)d}{c}\right) 
\end{align*}
where the summation symbol $\sum_*$ is indexed according to the double coset decomposition \cite{Sel}. The trick is to take advantage of the probable cancellation of terms due to variations in the argument by working directly with
$$
\sum_{0<c\leq x} \frac{S(m,n\ ; c,\chi_k)}{c}.
$$ 
The study of sums of Kloosterman sums has a long history, which goes well beyond the scope of our purpose. It is sufficient for us to know that the above average has an expansion of the form
\begin{align}\label{GS estimate}
\sum_{0<c\leq x} \frac{S(m,n\ ; c,\chi_k)}{c}\ =\ \sum_{j=1}^l \tau_j x^{\alpha_j} + O\left( x^{\beta/3+\eps}\right)\qquad (\eps>0)
\end{align}
for certain constants $\tau_j$, $\alpha_j\in(0,1)$ and $\beta \leq 1$ \cite[Thm. 2]{GS}.

We note that both Selberg and Goldfeld--Sarnak assumed that $\Gamma$ is a finite index subgroup of $\SL(2,\Z)$, but, as noted in \cite{GS}, there is no need to make that restriction.
 
\section{Proof of Theorem \ref{Dedekind symbols}}

\subsection*{Part (1)}

Let $\Gamma<\SL(2,\R)$ be a non-cocompact lattice. Fix a cusp $\frak{a}$ for $\Gamma$. If $\frak{a}$ is not equivalent to $\infty$, fix a scaling matrix $\sigma_\frak{a}$. Observe that
$$
\vol(\Gamma\backslash\h)\ =\ \vol(\sigma_\frak{a}^{-1}\Gamma\sigma_\frak{a}\backslash\h)\ =:\ V
$$
and
$$
E_\frak{a}(\sigma_\frak{a}z,s;\Gamma)\ =\ E_\infty(z,s;\sigma_\frak{a}^{-1}\Gamma\sigma_\frak{a})\ =:\ E(z,s)
$$
-- note that $E(z,s)$ does not depend on the choice of the cusp representative $\frak{a}$ nor of the scaling matrix $\sigma_\frak{a}$. The constant term of the Eisenstein series in the Laurent expansion at its first pole $s=1$,
\begin{align*}
E(z,s)\ =\ \frac{V^{-1}}{s-1} + K(z) + O(s-1) \qquad(s\to1)
\end{align*}
is called the first-order Kronecker limit function $K(z)$. Using the Fourier expansion of the Eisenstein series, Jorgenson and O'Sullivan derive that 
\begin{align}\label{Fourier exp KL}
K(z)\ =\ \sum_{n<0}k(n)e(n\overline{z}) + y + k(0) - V^{-1}\ln y + \sum_{n>0}k(n)e(nz)
\end{align}
and prove that the constants $k(n)$ satisfy $k(-n)=\overline{k(n)}$ and $k(n)\ll \abs{n}^{1+\eps}$ with an implied constant depending only on $\Gamma$ and $\eps>0$ \cite[Thm. 1.1]{JO}. In particular, the Kronecker limit function is well-defined. Moreover, it is a $\sigma_\frak{a}^{-1}\Gamma\sigma_\frak{a}$-invariant function, and can be seen to be real-valued and real-analytic if we set
\begin{align*}
K(z)\ =\ \lim_{\substack{s\to 1 \\ s\in\R_{>1}}} \left( E(z,s) - \frac{V^{-1}}{s-1}\right).
\end{align*}
A simple computation yields $$\Delta K(z) = \frac{-1}{V}$$ and from that observation we construct the harmonic function
\begin{align*}
H(z)\ :=\ V K(z) + \ln\im(z).
\end{align*}
Let $F:\h\to\C$ denote the holomorphic function with real part $\re F(z)=H(z)$. Observe that $F$ won't be automorphic, as the perturbation from $K$ to $H$ induces the logarithmic defect
\begin{align*}
H(z) - H(\gamma z)\ =\ \ln \abs{cz+d}^2.
\end{align*}
By analogy with Dedekind's transformation formula for $\log\eta$, we want to consider the RHS as the real part of the {\it principal} branch of logarithm, that is, the branch with $-\pi<\arg(z)\leq\pi$.  Hence
$$
\re\left( F(z) -F(\gamma z)\right)\ =\ \re\log(-(cz+d)^2).
$$
As a consequence of the Open Mapping Theorem, the difference
$$
\phi_{\frak{a}}(\gamma) := \frac{1}{2\pi i}\left( \frac{F(z)}{2} - \frac{F(\gamma z)}{2} -\frac{\log\left(-(cz+d)^2\right)}{2}\right)
$$
is a real-valued function of $\sigma_\frak{a}^{-1}\Gamma\sigma_\frak{a}$ that does not depend on $z$. Finally, note that $\phi_\frak{a}(-\gamma)=\phi_\frak{a}(\gamma)$ for all $\gamma\in\sigma_\frak{a}^{-1}\Gamma\sigma_\frak{a}$.

Define, for each scalar $k\in\R$, the function 
$$
f_{\frak{a},k}(z)\ =\ e^{-kF(z)/2}.
$$
It is holomorphic, nowhere vanishing, with $\log f_{\frak{a},k}(z) = -kF(z)/2$ and
$$
\log f_{\frak{a},k}(\gamma z)\ =\ \log f_{\frak{a},k}(z) +\frac{k}{2}\log\left(-(cz+d)^2\right) + 2\pi ik\phi_\frak{a}(\gamma)
$$
for all $\gamma\in\Gamma$, $z\in\h$. Finally, observe that 
\begin{align*}
\log(-(cz+d)^2)\ &=\ 2\left( \log(cz+d) -\frac{\pi i}{2}\sign(c(-d))\right)\\ &=\ 2\log\left(\frac{cz+d}{i\sign(c(-d))}\right).
\end{align*}
where the symbol $c(-d)=c$ if $c\neq 0$ and $-d$ otherwise. 

\subsection*{Part (2)}\label{phi section}

Let $\Gamma<\SL(2,\R)$ be a non-cocompact lattice. Fix a cusp $\frak{a}$ for $\Gamma$. If $\frak{a}$ is not equivalent to $\infty$, fix a scaling matrix $\sigma_\frak{a}$ as before. We define the Dedekind symbol for $\Gamma$ at its cusp $\frak{a}$ as the double coset function $\mathcal{S}_\frak{a}:\Gamma_\frak{a}\backslash\Gamma\slash\Gamma_\frak{a}\to\R$, given by
$$
\mathcal{S}_\frak{a}\left(\Gamma_\frak{a}\gamma\Gamma_\frak{a}\right)\ =\ \begin{cases} \frac{V}{4\pi}\frac{a+d}{c} -\phi_\frak{a}\bpm a&b\\ c&d\epm & c\neq 0 \\ \frak{a} & c=0 \end{cases}
$$
where 
$$
\bpm a&b\\c&d\epm\ =\ \sigma_\frak{a}^{-1}\gamma\sigma_\frak{a}
$$
and $\phi_\frak{a}:\sigma_\frak{a}^{-1}\Gamma\sigma_\frak{a}\to\R$ as in the proof of Part (1). We will show that $\mathcal{S}_\frak{a}$ is indeed well-defined. 

First, this definition does not depend on the choice of the scaling matrix $\sigma_\frak{a}$ associated to $\frak{a}$. In fact, going through the construction of the function $\phi_\frak{a}$ from Part (1) for the groups $\sigma_\frak{a}^{-1}\Gamma\sigma_\frak{a}$ and $\sigma_\frak{a}'^{-1}\Gamma\sigma_\frak{a}'$ where $\sigma_\frak{a}' =\sigma_\frak{a}n_x$, $n_x =\bsm 1&x\\&1\esm$, yields
$$
\phi_{\frak{a},\sigma_\frak{a}^{-1}\Gamma\sigma_\frak{a}}(n_x^{-1}\gamma n_x)\ =\ \phi_{\frak{a},\sigma'^{-1}_\frak{a}\Gamma\sigma_\frak{a}'}(\gamma)
$$
for all $\gamma\in\sigma_\frak{a}^{-1}\Gamma\sigma_\frak{a}$. 

To show that 
$$
\frac{V}{4\pi}\frac{a+d}{c} - \phi_\frak{a}\bpm a&*\\c&d\epm
$$ 
is well defined on double cosets $T\bsm a&*\\c&d\esm T$, it suffices to prove that, for any $m,n\in\Z$, 
\begin{align}\label{show this}
\phi_\frak{a}\left(\bpm1&m\\&1\epm \gamma \bpm 1&n\\&1\epm\right)\ =\ \phi_\frak{a}(\gamma) +(m+n)\frac{V}{4\pi}
\end{align}
For $k=1$, let $f_\frak{a}\equiv f_{\frak{a},1}$. We have
$$
\log f_\frak{a}(\gamma z) -\log f_\frak{a}(z)\ =\ \log(cz+d) +2\pi i\left(\phi_\frak{a}(\gamma) -\frac{\sign(c(-d))}{4}\right).
$$
Consider the associated function
\begin{align*}
\psi_\frak{a}(\gamma)\ :=\ \phi_\frak{a}(\gamma) - \frac{1}{4}\sign(c(-d))
\end{align*}
on $\sigma_\frak{a}^{-1}\Gamma\sigma_\frak{a}$. (Note that contrarily to $\phi_\frak{a}$, $\psi_\frak{a}$ is not well defined on $\sigma_\frak{a}^{-1}\Gamma\sigma_\frak{a}\slash\{\pm I\}$.) Comparing 
\begin{eqnarray*}
\psi_\frak{a} (\gamma_1\gamma_2) & = & \frac{1}{2\pi i}\left( \log\eta_\frak{a}(\gamma_1\gamma_2 z) - \log\eta_\frak{a} (z) -  \log j(\gamma_1\gamma_2,z)\right) \\
\psi_\frak{a} (\gamma_1) & = & \frac{1}{2\pi i}\left( \log\eta_\frak{a} (\gamma_1\gamma_2 z) - \log\eta_\frak{a} (\gamma_2 z) -  \log j(\gamma_1,\gamma_2 z)\right) \\
\psi _\frak{a}(\gamma_2) &  = & \frac{1}{2\pi i}\left( \log\eta_\frak{a} (\gamma_2 z) - \log\eta_\frak{a} (z) - \log j(\gamma_2,z)\right)
\end{eqnarray*}
yields the coboundary
$$
\psi_\frak{a}(\gamma_1\gamma_2) - \psi_\frak{a}(\gamma_1) - \psi_\frak{a}(\gamma_2)\ =\ \frac{1}{2\pi i}\left(\log j(\gamma_1,\gamma_2 z) + \log j(\gamma_2,z) -\log j(\gamma_1\gamma_2,z)\right).
$$
Using the cocycle relation 
$$
j(\gamma_1\gamma_2,z)\ =\ j(\gamma_1,\gamma_2 z)j(\gamma_2,z),
$$
 we observe that if either $\gamma_1$ or $\gamma_2$ is of the form $\bsm1&m\\&1\esm$, then the RHS of the coboundary equation above vanishes. Hence, as $c$ in $T\bsm a&*\\c&d\esm T$ is uniquely determined,
$$
 \psi_\frak{a}\left(\bpm 1&m\\&1\epm\gamma\bpm 1&n\\&1\epm\right)\
 =\  \psi_\frak{a}\bpm1&m\\&1\epm + \psi_\frak{a}\bpm1&n\\&1\epm +\psi_\frak{a}(\gamma).
 $$
We next determine 
$$
\psi_\frak{a}\bpm 1&m\\&1\epm.
$$

The expansion (\ref{Fourier exp KL}) for the Kronecker limit function $K(z)$ is equivalent to
$$
\re\left(-\log f_\frak{a}(z)\right)\ =\ \frac{V}{2}\re\left(\frac{z}{i} +k(0) +2\sum_{n>0} k(n)e(nz)\right).
$$
Denote the expression in parenthesis on the RHS by $U(z)$. By yet another application of the Open Mapping Theorem, 
\begin{align*}%\label{2}
\log f_\frak{a}(z) +U(z)\ =\  \log f_\frak{a}(z+m) + U(z+m)
\end{align*}
and observe that 
\begin{align*}%\label{3}
U(z) - U(z+m)\ =\ i m \frac{V}{2}.
\end{align*}
Then
$$
\psi_\frak{a}\bpm 1&m\\&1\epm\ =\ \frac{1}{2\pi i}\left(\log f_\frak{a}(z+m) - \log f_\frak{a}(z)\right)\ =\ m \frac{V}{4\pi}.
$$
Equation (\ref{show this}) follows.

\subsection*{Part (3)}
We run through the proof of part (1) replacing the definition of the harmonic function $H(z)$ by
$$
H(z)\ =\ \frac{\pi}{3} K(z) + \ln(y) - c +\ln(4\pi).
$$
Here, $c$ denotes the Euler constant. Note that this perturbation by a constant has no impact on the properties of $H(z)$ that are necessary in the proof (harmonicity, transformation formula). Then, by Kronecker's first limit formula 
$$
K(z)\ =\ \frac{3}{\pi}\left( c-\ln(4\pi) -\ln\left( y\abs{\eta(z)}^4\right)\right),
$$
we have $H(z) = -4\ln\abs{\eta(z)}$. By the usual Open Mapping argument, there exists an imaginary constant $\alpha\in i\R$ such that $F(z) = -4\log\eta(z) +\alpha$, with $F(z)$ the holomorphic function such that $\re(F(z))=H(z)$. It now follows from our definitions that
$$
\phi(\gamma) = \frac{2}{2\pi i}\left( \log\eta(\gamma z) -\log\eta(z) -\frac{1}{2}\log\left(\frac{cz+d}{i\sign(c)}\right)\right) \overset{(\ref{DTF})}{=} \frac{\Phi(\gamma)}{12}.
$$

\section{Proof of Corollary \ref{cusp form thm}}

Let $\Gamma<\SL(2,\R)$ be a non-cocompact lattice. Let $\frak{a},\frak{b},\dots,\frak{n}$ be all inequivalent cusps for $\Gamma$. By definition of $f_{\frak{a},k}$,
$$
\abs{f_{\frak{a},k}(\sigma_\frak{b}z)}\ =\ e^{-\frac{k}{2}H_\frak{a}(\sigma_\frak{b} z)}\ =\ \abs{j(\sigma_\frak{b}, z)^k} e^{-\frac{k}{2}\left(VK_\frak{a}(\sigma_\frak{b} z) +\ln y\right)}.
$$ 
By \cite[Eq. (4.7)]{JO}, $K_\frak{a}(\sigma_\frak{b} z) + V^{-1}\ln y\ =\ \delta_\frak{ab}y + k_\frak{ab}(0) + O\left( e^{-2\pi y}\right).$ Hence,
$$
j(\sigma_b,z)^{-k} f_{\frak{a},k}(\sigma_\frak{b}z)\ \ll\ e^{-\frac{kV}{2}\delta_\frak{ab} y}
$$
as $y\to\infty$. In particular, for any fixed positive real weight $k$, $f_{\frak{a},k}$ decays exponentially in the cusp $\frak{a}$ and only in that cusp. Set
$$
\eta_{\Gamma,K}\ =\ f_{\frak{a},k_1}f_{\frak{b},k_2}\cdots f_{\frak{n},k_n}
$$
for $k_1,k_2,\dots,k_n>0$, $K:= k_1+k_2+\dots + k_n$. This function is holomorphic, nowhere vanishing and decays exponentially in every cusp. In fact, for each cusp $\frak{j}$,
$$
j(\sigma_\frak{j},z)^{-K}\eta_{\Gamma,K}(\sigma_\frak{j}z)\ \ll\ e^{-\frac{k_j V}{2} y}
$$
as $y\to\infty$.

\section{Proof of Theorem \ref{counting thm}}

We define the zeta function
\begin{align*}%\label{Dirichlet series}
Z(s)\ :=\ \sum_{c>0} \frac{a_c}{c^{2s}}\qquad (s=\sigma+it\in\C)
\end{align*}
with coefficients $(a_c)_{c>0}$ given by 
$$
a_c\ =\ \#\left\{ 0\leq a<c : \bpm a&*\\c&*\epm\in\sigma_\frak{a}^{-1}\Gamma\sigma_\frak{a}\right\}.
$$
The series is absolutely convergent for $\sigma>1$, as $a_c\ll c^2$ \cite[Prop. 2.8]{Iwa}.  By an integration by parts argument, for any $s$ with $\sigma>1$,
\begin{align*}
Z(s)\ &=\ \lim_{x\to\infty} \sum_{c\leq x} \frac{a_c}{c^{2s}}\\ &=\ \lim_{x\to\infty}\left(\frac{\pi(x)}{x^{2s}} + 2s\int^x_0\pi(u)u^{-2s-1}du\right)\ =\ 2s\pi^\wedge(2s),
\end{align*}
where $\pi^\wedge(\cdot)$ denotes the Mellin transform of $\pi(\cdot)$.

By the Mellin Inversion Theorem, we recover the Perron formula
$$
\pi(x)\ =\ \frac{1}{2\pi i}\int_{(\sigma)} Z(s) \frac{x^{2s}}{s}ds\ :=\ \lim_{T\to\infty}\frac{1}{2\pi i}\int_{\sigma-iT}^{\sigma+iT} Z(s)\frac{x^{2s}}{s}ds
$$
again for $\sigma>1$. To obtain an asymptotic growth rate for $\pi(x)$, one can apply the effective Perron formula \cite[Thm. 2.3]{Ten} to truncate the RHS above to
\begin{align}\label{EPF}
\pi(x)\ =\ \frac{1}{2\pi i}\int_{\sigma-iT}^{\sigma+iT} Z(s)\frac{x^{2s}}{s}ds + O\left( \frac{x^{2\sigma}}{T} \right)
\end{align}
where $T$ is a positive, large parameter, which we will later choose depending on $x$. 

The function $\varphi_\frak{a}\equiv \varphi_\frak{aa}$ appearing in the Fourier expansion of the Eisenstein series is expressed in terms of $Z(s)$. In fact,
\begin{align*}
Z(s)\ =\ \frac{\pi^{-1/2}\Gamma(s)}{\Gamma(s-1/2)}\ \varphi_\frak{a}(s)
\end{align*}
whenever $\sigma>1$. This allows to meromorphically continue $Z(s)$ to the half-plane $\sigma>1/2$, where the poles are those of $\varphi_\frak{a}(s)$. That is, $Z(s)$ has possibly finitely many simple poles $1/2<\sigma_j<1$ and a simple pole at $s=1$. The residue of the latter is given by 
\begin{align*}
\underset{s=1}{\res}\ Z(s)\ =\  \frac{1}{\pi\vol(\Gamma\backslash\h)}.
\end{align*}
 Moreover, because $\varphi_\frak{a}(s)$ is bounded in the half-plane $\sigma>1/2$, $Z(s)$ may be there approximated directly from Stirling's formula. That is, explicitly,  
\begin{align*}
Z(s)\ \ll\ \pi^{-1/2}\abs{s-1/2}^{1/2}\ \ll \abs{t}^{1/2}.
\end{align*}
We apply the Residue Theorem to the rectangular path of integration with vertices $\frac{1}{2}+\eps\pm iT$, $1+\eps\pm iT$.
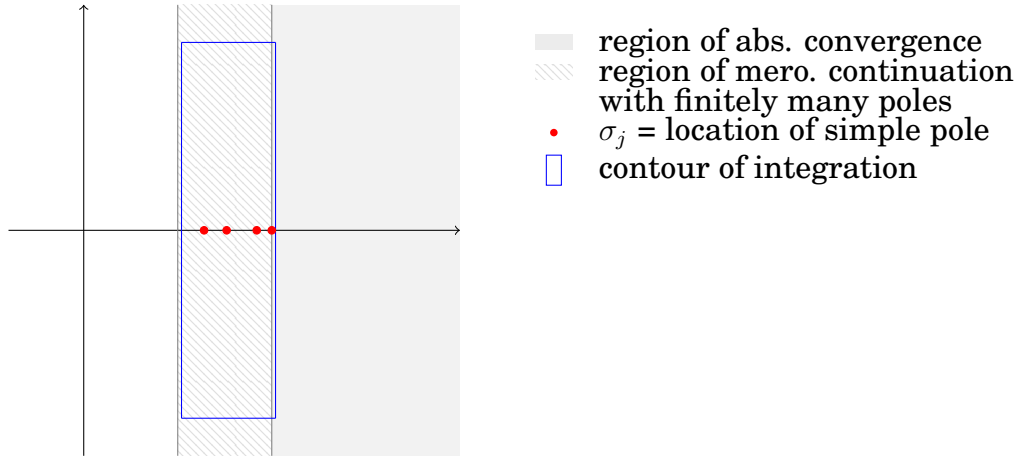
\begin{center}
\begin{tikzpicture}
\fill[color=gray!10] (2.5,-3)--(2.5,3)--(5,3)--(5,-3)--cycle;
\fill[thick,pattern=north west lines, pattern color=gray!25] (1.25,-3)--(1.25,3)--(2.5,3)--(2.5,-3)--cycle;
\draw[->] (-1,0) -- (5,0); 
\draw[->] (0,-3) -- (0,3);
\draw[blue] (2.55,-2.5)--(2.55,2.5);
\draw[gray] (2.5,-3)--(2.5,3);
\draw[blue] (1.3,-2.5)--(1.3,2.5);
\draw[gray] (1.25,-3)--(1.25,3);
\draw[blue] (2.55,-2.5)--(1.3,-2.5);
\draw[blue] (2.55,2.5)--(1.3,2.5);
\draw[fill,red] (1.6,0) circle[radius=0.05];
\draw[fill,red] (1.9,0) circle[radius=0.05];
\draw[fill,red] (2.3,0) circle[radius=0.05];
\draw[fill,red] (2.5,0) circle[radius=0.05];
\fill[gray!15] (6,2.6)--(6.5,2.6)--(6.5,2.4)--(6,2.4)--cycle;
\draw (6.7,2.5) node [right] {region of abs. convergence};
\fill[thick,pattern=north west lines, pattern color=gray!25] (6,2.2)--(6.5,2.2)--(6.5,2)--(6,2)--cycle;
\draw (6.7,2.1) node [right] {region of mero. continuation};
\draw (6.7,1.7) node [right] {with finitely many poles};
\fill [red] (6.25,1.3) circle[radius=0.05];
\draw (6.7,1.3) node [right] {$\sigma_j$ = location of simple pole};
\draw[blue] (6.15, 0.6)--(6.15,1.0)--(6.35,1.0)--(6.35,0.6)--cycle;
\draw (6.7,0.8) node [right] {contour of integration};
\end{tikzpicture}
\end{center}

The contribution on the left vertical segment is bounded by
$$
x^{1+\eps} \int^{T}_{-T} \abs{t}^{-1/2} dt\ =\ x^{1+\eps}T^{1/2},
$$
the contribution of the horizontal segments by
$$
T^{-1/2}\int^{1+\eps}_{1/2+\eps} x^{2\sigma}d\sigma\ \ll\ T^{-1/2}x^{2+\eps},
$$
for $x$ large enough, so that with (\ref{EPF}),
\begin{eqnarray*}
\pi(x)\ &=\ &  \sum_{1/2<\sigma_j\leq 1} \underset{s=\sigma_j}{\res}(Z(s))\frac{x^{2\sigma_j}}{\sigma_j} + O\left(x^{1+\eps}T^{1/2} + \frac{x^{2+\eps}}{T^{1/2}}+\frac{x^{2+\eps}}{T}\right) \\
& =\ & \frac{x^2}{\pi\vol(\Gamma\backslash\h)} + \sum_{1/2<\sigma_j <1} \underset{s=\sigma_j}{\res}(Z(s))\frac{x^{2\sigma_j}}{\sigma_j} + O\text{-term}.
\end{eqnarray*}
The error term is minimised by choosing $T=x$, yielding $O\left(x^{3/2+\eps}\right).$

We would like to conclude with some remarks on the (non-)optimality of the above estimate. The analytical proof presented here, while sufficient for the equidistribution problem at hand, is not well suited to obtain an optimal error term. To see this, take $\Gamma=\SL(2,\Z)$. Then the counting function $\pi(x)$ is precisely the partial sum
$$
\pi(x)\ =\ \sum_{n=1}^{\lfloor x\rfloor}\sum_{\substack{ a=1 \\ (a,n)=1}}^n 1\ =\ \sum_{n=1}^{\lfloor x\rfloor} \phi_\text{tot}(n)
$$
for the Euler totient function $\phi_\text{tot}$. On the other hand, we have
$$
Z(s)\ =\ \sum_{n\geq1}\frac{\phi_\text{tot}(n)}{n^{2s}}\ =\ \frac{\zeta(2s-1)}{\zeta(2s)},
$$
where $\zeta$ is the Riemann $\zeta$-function. Then $Z(s)$ has no poles in $[1/2,1)$. Upon assuming the Riemann Hypothesis, there are no poles in $(1/4,1)$ and we can improve our estimate to
$$
\pi(x)\ =\ \frac{3}{\pi^2}x^2 + O\left(x^{5/4+\eps}\right).
$$
However, by way of algebraic manipulations, we have the well-known elementary estimate
$$
\pi(x)\ =\ \frac{3}{\pi^2}x^2 + O(x\ln x).
$$
The error term here is already much stronger, and yet still far from optimal ;  Montgomery conjectured for the maximum order of magnitude of the remainder term
$$
R(x)\ =\ \pi(x) - \frac{3}{\pi^2}x^2
$$
that $R(x)\ll x\ln\ln x$ should hold \cite{Mont}.

\section{Proof of Theorem \ref{equidistribution mod 1}}

Unfolding the definition of the Dedekind symbol $\mathcal{S}_\frak{a}$ yields, for each $k\in\R_{>0}$,
$$
\sum_{0<c\leq x}\sum_* e\left(k\mathcal{S}_\frak{a}\left(\Gamma_\frak{a}\gamma\Gamma_\frak{a}\right)\right)\ =\ e\left(-\frac{k}{4}\right)\sum_{0<c\leq x}\sum_* \overline{\chi_{\frak{a},k}\bpm a&*\\c&d\epm } e\left( \frac{kV}{4\pi}\frac{a+d}{c}\right)
$$
where the summation symbol $\sum_*$ is again indexed on the double coset decomposition, 
$$
\bpm a&b\\c&d\epm\ =\ \sigma_\frak{a}^{-1}\gamma\sigma_\frak{a}
$$
and 
\begin{align*}%\label{MS}
\chi_{\frak{a},k}(\cdot) = e^{2\pi ik\psi_\frak{a}(\cdot)}
\end{align*}
 defines a multiplier system of weight $k$ for $\sigma_\frak{a}^{-1}\Gamma\sigma_\frak{a}$, determined by the function $\psi_\frak{a}$ constructed in Section 3. 
 \begin{lm} 
 For the multiplier system $\chi_{\frak{a},k}$, 
$$
\alpha\left(\chi_{\frak{a},k}\right)\ =\ \left\lceil \frac{k\ \vol(\Gamma\backslash\h)}{4\pi}\right\rceil -\frac{k\ \vol(\Gamma\backslash\h)}{4\pi},
$$
where $\lceil x\rceil$ denotes the smallest integer $\geq x$. 
\end{lm}
\begin{proof} This follows from $\psi_\frak{a}\bsm1&1\\&1\esm = \vol(\Gamma\backslash\h)/4\pi$. \end{proof}

\begin{prop}\label{prop Vardi identity}
In the context of Dedekind symbols, Vardi's identity translates to
$$
\sum_{0<c\leq x}\sum_* e\left(k\cdot\mathcal{S}_\frak{a}\left([\gamma]\right)\right)\ =\ e\left(-\frac{k}{4}\right)\sum_{0<c\leq x} S\left(\left\lceil\frac{kV}{4\pi}\right\rceil,\left\lceil\frac{kV}{4\pi}\right\rceil\ ;c, \chi_{\frak{a},k}\right)
$$
where $k\in\R_{>0}$, $\mathcal{S}_\frak{a}$ is the Dedekind symbol for $\Gamma$ at its cusp at $\frak{a}$, and $S$ is the general Kloosterman sum for $\sigma_\frak{a}^{-1}\Gamma\sigma_\frak{a}$ with respect to the multiplier system $\chi_{\frak{a},k}$.
\end{prop}

We integrate by parts the Goldfeld--Sarnak estimate (\ref{GS estimate}) and obtain as a result
$$
\sum_{c\leq x}\sum_* e\left( k\cdot\mathcal{S}([\gamma])\right)\ =\ e\left(-\frac{k}{4}\right)\left(\sum_{j=1}^l \tau_j x^{1+\alpha_j} + O\left(x^{1+\beta/3+\eps}\right)\right)
$$
where $0<\alpha_j<1$ and $\beta\leq 1$. In particular, the exponent of $x$ is smaller than 2. Hence, by Theorem \ref{counting thm} and Weyl's equidistribution criterium, the sequence $\{k\mathcal{S}_\frak{a}(\cdot)\}$ becomes equidistributed mod 1 as $c\to\infty$. 

We can moreover obtain a quantitative rate of this equidistribution. For any subinterval $[a,b]\subset[0,1]$, set 
$$
R(x)\ =\ \frac{\#\{ \text{first }\pi(x)\text{ terms in }\{ k\mathcal{S}(\cdot)\}\cap[a,b]}{\pi(x)} - (b-a).
$$
Then, by the Erd\"os--Tur\'an inequality \cite{ET},
$$
R(x)\ \ll\ \frac{1}{M} +\left( \sum_{m=1}^M \frac{e^{-\pi i\frac{mk}{2}}}{m}\right)\left( \sum_{j=1}^l \widetilde{\tau}_j x^{\alpha_j -1} +O\left(x^{-2/3 +\eps}\right)\right)
$$
for any positive integer $M$.

%\subsection*{Acknowledgements}
%
%The author wishes to thank Marc Burger and \'Arp\'ad T\'oth for their comments on early drafts of this paper. Many thanks to the anonymous referees, who substantially helped improve the quality of exposition. I more particularly wish to thank the referee who pointed out the existence of the Erd\"os--Tur\'an inequality. This is part of the author's PhD thesis at ETH Z\"urich. Substantial parts of this paper were written while the author was in residence at MSRI in Spring 2015.

\noindent \textsc{Department Mathematik, ETH Zentrum, 8092 Z\"urich}\\
{\em E-mail address: } claire.burrin@math.ethz.ch

\end{document}